\documentclass[review]{elsarticle}

\usepackage{lineno,hyperref}

\usepackage{graphicx}% Include figure files
\usepackage{dcolumn}% Align table columns on decimal point
\usepackage{bm}% bold math

\usepackage{graphics} % for pdf, bitmapped graphics files
\usepackage{epsfig} % for postscript graphics files
\usepackage{times} % assumes new font selection scheme installed
\usepackage{amsmath} % assumes amsmath package installed
\usepackage{amssymb}  % assumes amsmath package installed

\usepackage{amsfonts}
\usepackage{fancyhdr}
\usepackage{color}
\usepackage{subfigure}

\newtheorem{mytheorem}{Theorem}

\newtheorem{mylemma}[mytheorem]{Lemma}

\newtheorem{remark}{Remark}

\newtheorem{problem}{Problem}
\newproof{proof}{Proof}

\modulolinenumbers[5]

\journal{xxxx}

\begin{document}

%%%%%%%%%%%%%%%%%%%%%%%%%%%%%%%%%%%%%%%%%%%%%%
\begin{frontmatter}

\title{In-domain control of a heat equation: an approach combining zero-dynamics inverse and differential flatness}
%\tnotetext[mytitlenote]{Fully documented templates are available in the elsarticle package on %\href{http://www.ctan.org/tex-archive/macros/latex/contrib/elsarticle}{CTAN}.}


\author[mymainaddress,mysecondaryaddress]{Jun Zheng}
\ead{zhengjun2014@aliyun.com}

\author[mysecondaryaddress]{Guchuan Zhu\corref{mycorrespondingauthor}}
\cortext[mycorrespondingauthor]{Corresponding author \\ \hspace*{12pt} Tel: +1-514-340-4711 ext. 5868}
\ead{guchuan.zhu@polymtl.ca}


\address[mymainaddress]{Department of Basic Courses, Southwest Jiaotong University, Emeishan, Sichuan 614202, China}
\address[mysecondaryaddress]{Department of Electrical Engineering, Polytechnique Montreal, P.O. Box 6079, Station Centre-Ville, Montreal, QC, Canada H3T 1J4}

%%%%%%%%%%%%%%%%%%%%%%%%%%%%%%%%%%%%%%%%%%%%%%
\begin{abstract}
This paper addresses the set-point control problem of a heat equation with in-domain actuation. The proposed scheme is based on the framework of zero dynamics inverse combined with flat system control. Moreover, the set-point control is cast into a motion planing problem of a multiple-input, multiple-out system, which is solved by a Green's function-based reference trajectory decomposition. The validity of the proposed method is assessed through the convergence and solvability analysis of the control algorithm. The performance of the developed control scheme and the viability of the proposed approach are confirmed by numerical simulation of a representative system.
\end{abstract}


\end{frontmatter}
%%%%%%%%%%%%%%%%%%%%%%%%%%%%%%%%%%%%%%%%%%%%%%

\linenumbers

%%%%%%%%%%%%%%%%%%%%%%%%%%%%%%%%%%%%%%%%%%%%%%

%%%%%%%%%%%%%%%%%%%%%%%%%%%%%%%%%%%%%%%%%%%%%%%%%%%%%%%%%%%%%%%%%%%%%%%%%%%%%%%%
\section{Introduction}\label{Sec: Introduction}
Control of parabolic partial differential equations (PDEs) is a long-standing problem in PDE control theory and practice. There exists a very rich literature devoted to this topic and it is continuing to draw a great attention for both theoretical studies and practical applications. In the existing literature, the majority of work is dedicated to boundary control, which may be represented as a standard Cauchy problem to which functional analytic setting based on semigroup and other related tools can be applied (see, e.g., \cite{Bensoussan:2006,Zwart:95,Lasiecka:2000,Tucsnak2009}). It is interesting to note that in recent years, some methods that were originally developed for the control of finite-dimensional systems have been successfully extended to the control of parabolic PDEs, such as backstepping (see, e.g., \cite{Krstic:2008-1,Smyshyaev:2010,Tsubakino:2012}), flat systems (see, e.g., \cite{Kharitonov:2006,Laroch2000,Lynch2005,Meurer:2013,Petit2002,Rudolph2003,Schorkhuber:2013}), as well as their variations (see, e.g., \cite{Malchow:2011,Meurer:2011,Meurer:2009}).

This paper addresses the in-domain (or interior point) control problem of a heat equation, which may arise in application related concerns for, e.g., the enhancement of control efficiency. Integrating a number of control inputs acting in the domain will lead to non-standard inhomogeneous PDEs \cite{Bensoussan:2006, Lasiecka:2000-2}, which should be treated differently than the standard boundary control problem. The control scheme developed in the present work is based-on the framework of zero-dynamics inverse (ZDI), which was introduced by Byrnes and his collaborators in \cite{BG88} and has been exploited and developed in a series of work \cite{BG09,BGIS03,BGIS206,BGIS106,BGIS13}. It is pointed out in \cite{BGIS106} that ``for certain boundary control systems it is very easy to model the system's zero dynamics, which, in turn, provides a simple systematic methodology for solving certain problems of output regulation." Indeed, the construction of zero dynamics for output regulation of certain in-domain controlled PDEs is also straightforward (see, e.g., \cite{BGIS206}) and hence, the control design can be carried out in a systematic manner. A main issue related to the ZDI design is that the computation of dynamic control laws requires resolving the corresponding zero dynamics, which may be very difficult for generic regulation problems, such as set-point control. To overcome this difficulty, we leverage one of the fundamental properties of flat systems, that is if a lumped or distributed parameter system is differentially flat (or flat for short), then its states and inputs can be explicitly expressed by the so-called flat output and its time-derivatives \cite{Fliess95,Rudolph2003}. In the context of ZDI design, the control can be derived from the flat output without explicitly solving the zero dynamics. Moreover, in the framework of flat systems, set-point control can be cast into a problem of motion planning, which can also be carried out in a systematic manner.

The system model used in this work is taken from \cite{BGIS206}. In order to perform control design, we present the original system in a form of \emph{parallel connection}. This formulation allows a significant simplification of computation. As the control with multiple actuators located in the domain leads to a multiple-input, multiple-output (MIMO) problem, the design of reference trajectories is not trivial. To overcome this problem, we introduce a Green's function-based reference trajectory decomposition scheme that enables a simple and computational tractable implementation of the proposed control algorithm.

The remainder of the paper is organized as follows. Section~\ref{Sec: Problem Setting} describes the model of the considered system and its equivalent settings. Section~\ref{Sec: Control design} presents the detailed control design. Section~\ref{Sec: Motion Planning} deals with motion planning and addresses the convergence and the solvability of the proposed control scheme. A simulation study is carried out in Section~\ref{simulation}, and, finally, some concluding remarks are presented in Section~\ref{Sec: Conclusion}.

%%%%%%%%%%%%%%%%%%%%%%%%%%%%%%%%%%%%%%%%%%%%%%%%%%%%%%%%%%%%%%%%%%%%%%%%%%%%%%%%
\section{Problem Setting}\label{Sec: Problem Setting}
In the present work, we consider a scaler parabolic equation describing one-dimensional heat transfer with
boundary and in-domain control, which is studied in \cite{BGIS206}. Denote by $z(x,t)$ the heat distribution over
the one-dimensional space, $x$, and the time, $t$. The derivatives of $z(x,t)$ with respect to its variables
are denoted by $z_{x}$ and $z_{t}$, respectively. For notational simplicity, we may not show all the variables
of functions if there is no ambiguity, e.g., $z=z(x,t)$. Consider $m$ points $x_j$, $j = 1,\ldots, m$,
in the interval $(0,1)$ and assume, without loss of generality, that $0=x_0<x_1<x_2<\cdots<x_m<x_{m+1}=1$. Let
$\Omega \doteq \underset{j=0}{\overset{m}{\bigcup}}(x_{j},x_{j+1})$. The considered heat equation with boundary
and in-domain control in a normalized coordinate is of the form:
\begin{subequations}\label{eq: heat eq}
\begin{align}
  &z_{t}(x,t)-z_{xx}(x,t)=0,\ \ x\in \Omega,\ t>0, \label{eq: heat eq_dynm}\\
  &z(x,0)=\phi(x),  \label{eq: heat eq_IC} \\
  &B_{0}z=z_{x}(0,t)-k_{0}z(0,t)=0,\label{eq: heat eq_B0} \\
  &B_{1}z=z_{x}(1,t)+k_{1}z(1,t)=0,\label{eq: heat eq_B1}\\
  &z(x_{j}^{+})=z(x_{j}^{-}),\ j=1,...,m,\\
  &B_{x_j}z=\left[z_{x}\right]_{x_j}=u_{j}(t),\ j=1,2,...,m,\label{eq: heat
  eq_jump}
\end{align}
\end{subequations}
where for a function $\psi(\cdot)$ and a point $x\in [0,1]$ we define
\begin{align*}
\left[\psi\right]_{x}=\psi(x^{+})-\psi(x^{-}),
\end{align*}
with $x^-$ and $x^+$ denoting, respectively, the usual meaning of left and
right hand limits to $x$. The initial condition is specified by \eqref{eq: heat eq_IC} with $\phi(x) \in L^2(0,1)$.
In System~\eqref{eq: heat eq}, we assume that we
can control the heat flow in and out of the system at the points
$x_{j}$, i.e.,
\begin{align}
  u_{j}(t)=\left[z_{x}\right]_{x_j}=z_{x}(x_{j}^{+},t)-z_{x}(x_{j}^{-},t).\notag
\end{align}
Note that in \eqref{eq: heat eq}, $B_{x_{j}}, x_{j} \in [0,1],$ represents the point-wise control located on the boundary or in the domain.

The space of weak solutions to System~\eqref{eq: heat eq} is chosen
to be $H^{1}(0,1)$. Note that System~\eqref{eq: heat eq} is
exponentially stable in $H^{1}(0,1)$ if the boundary controls $B_0$
and $B_1$ are chosen such that $k_0\geq 0$, $k_1\geq 0$, and $k_0 +
k_1 > 0$ \cite{BGIS106}.

 Denote a set of reference signals by $\{z^{D}_{i}(x_i,t)\}_{i=1}^{m}$. Let
\begin{align}
e_{i}(t)=z(x_{i},t)-z^{D}_{i}(x_i,t)\notag
\end{align}
be the regulation errors. Let
$e(t)=\{e_{i}(t)\}_{i=1}^{m}$ and $u(t)=\{u_{i}(t)\}_{i=1}^{m}$.
\begin{problem}
The considered regulation problem is to find a dynamic control $u(t)$
such that the regulation error satisfies
$e(t)\rightarrow 0$ as $t\rightarrow \infty$.
\end{problem}

 The above in-domain control problem can also be formulated in
another way by replacing the jump conditions in \eqref{eq: heat
eq_jump} by point-wise controls as source terms. The resulting
system will be of following the form
\begin{subequations}\label{eq: heat eq_2}
\begin{align}
  &z_{t}(x,t)-z_{xx}(x,t)=\sum_{j=1}^{m}\delta (x-x_{j})\alpha_{j}(t),\ 0< x< 1,\ t>0, \label{eq: heat eq_2_dynm}\\
  &z(x,0)=\phi(x),  \label{eq: heat eq_2_BC} \\
  &B_{0}z=z_{x}(0,t)-k_{0}z(0,t)=0,\label{eq: heat eq_2_B0} \\
  &B_{1}z=z_{x}(1,t)+k_{1}z(1,t)=0,\label{eq: heat eq_2_B1}
\end{align}
\end{subequations}
where $\delta(x-x_j)$ is the Dirac delta function supported at the
point $x_{j}$, denoting an actuation spot, and $\alpha_j : t
\mapsto \mathbb{R}$, $j = 1, \ldots, m$, are the in-domain control signals.
\begin{mylemma}\label{Prop: 1}
Considering weak solutions in $H^1(0,1)$, System \eqref{eq: heat eq}
and System~\eqref{eq: heat eq_2} are equivalent if
\begin{eqnarray}
\alpha_{j}(t)=-u_{j}(t)=-\left[z_{x}\right]_{x_j},\ \
j=1,...,m.\notag
\end{eqnarray}
\end{mylemma}
\begin{proof}
The proof follows the idea presented in \cite{Rebarber:1995}. Indeed, it suffices to prove ``System \eqref{eq: heat eq} $\Rightarrow$
System \eqref{eq: heat eq_2}." Let $X=L^{2}(0,1)$ be a Hilbert space
equipped with the inner product $\langle
u,v\rangle=\displaystyle\int_{0}^{1}u(x)v(x)\text{d}x$. Let the
operator $A$ be defined by $Au=u_{xx}$, with domain $\ \mathcal
{D}(A)=\{u\in H^{2}(0,1); B_{0}u=B_{1}u=0\}$. It is easy to see that
$A^{*}$, the adjoint of $A$, is equal to $A$. Let $\widetilde{A}$ be
an extension of $A$ with domain $\ \mathcal
{D}(\widetilde{A})=\{u\in X; \ u\in
H^{2}(\underset{i=0}{\overset{m}{\cup}}(x_{i},x_{i+1})),
 B_{0}u=B_{1}u=0,\ u(x_{j}^{+})=u(x_{j}^{-}),\ j=1,...,m
 \}$. Let $u\in \mathcal {D}(\widetilde{A}),\ v\in \mathcal {D}(A^*)=\mathcal
 {D}(A)$. Using integration by parts we obtain that
\begin{eqnarray}\label{+2.6}
\langle \widetilde{A}u,v\rangle = \langle u,Av\rangle
+\sum_{j=1}^{m}\left(u_{x}(x_{j}^{-})-u_{x}(x_{j}^{+})\right)v(x_{j}).
\end{eqnarray}
Let $X_{-1}=(\mathcal {D}(A^{*}))'$, the dual space of $\mathcal
{D}(A)$. We need to define another extension for $A$. Let $\hat{A}:
H^{1}(0,1) \rightarrow X_{-1}$ be defined by
\begin{align}\label{hat A}
\langle \hat{A}u,v\rangle=\langle u,A^*v\rangle\ \ \text{for\ all}\
v\in \mathcal {D}(A^*),
\end{align}
with $\mathcal {D}(\hat{A})=H^{1}(0,1)$. Note that $\delta
(\cdot-x_j)$ is not in $X$, but it is in the large space $X_{-1}$.
It follows from \eqref{+2.6}, \eqref{hat A}, and $A=A^*$ that
\begin{align}\label{delta A}
\widetilde{A}u=\hat{A}u+\sum_{j=1}^{m}\left(u_{x}(x_{j}^{-})-u_{x}(x_{j}^{+})\right)\delta
(x-x_j),
\end{align}
in $X_{-1}$. If $u$ satisfies System \eqref{eq: heat eq}, then
$\dot{u}(t)=\widetilde{A}u(t)$, which yields, considering \eqref{delta A},
$\dot{u}(t)=\hat{A}u+\sum_{j=1}^{m}(u_{x}(x_{j}^{-})-u_{x}(x_{j}^{+}))\delta
(x-x_{j})$. Finally, we can see that System \eqref{eq: heat eq}
becomes System \eqref{eq: heat eq_2} with
$\alpha_{j}(t)=-u_{j}(t)=-\left[z_{x}\right]_{x_j},\ j=1,...,m$,
where we look for generalized solutions $u(x,t)\in\mathcal
{D}(\hat{A})=H^1(0,1)$ such that \eqref{delta A} is true in
$X_{-1}$. \hfill $\Box$
\end{proof}

To establish in-domain control at every actuation point, we will
proceed in the way of \emph{parallel connection}, i.e., for every
$x_j\in (0,1)$, consider the following two systems
\begin{subequations}\label{eq: sub heat eq}
\begin{align}
  &z_{t}(x,t)-z_{xx}(x,t)=0,\ \ x\in (0,x_j)\cup (x_j,1),\ t>0, \label{eq: sub heat eq_dynm}\\
  &z(x,0)=\phi_{j}(x),  \label{eq: sub heat eq_BC} \\
  &B_{0}z=z_{x}(0,t)-k_{0}z(0,t)=0,\label{eq: sub heat eq_B0} \\
  &B_{1}z=z_{x}(1,t)+k_{1}z(1,t)=0,\label{eq: sub heat eq_B1}\\
  &z(x_{j}^{+})=z(x_{j}^{-}),\\
  &B_{x_j}z=\left[z_{x}\right]_{x_j}=v_{j}(t).\label{eq: sub heat eq_jump}
\end{align}
\end{subequations}
and
\begin{subequations}\label{eq: sub heat eq 2}
\begin{align}
  &z_{t}(x,t)-z_{xx}(x,t)=\delta (x-x_{j})\beta_{j}(t),\ 0< x< 1,\ t>0, \label{eq: sub heat eq_2_dynm}\\
  &z(x,0)=\phi_{j}(x),  \label{eq: sub heat eq_2_BC} \\
  &B_{0}z=z_{x}(0,t)-k_{0}z(0,t)=0,\label{eq: sub heat eq_2_B0} \\
  &B_{1}z=z_{x}(1,t)+k_{1}z(1,t)=0,\label{eq: sub heat eq_2_B1}
\end{align}
\end{subequations}
with $\sum_{j=1}^m\phi_{j}(x)=\phi(x)$. Similarly, System
\eqref{eq: sub heat eq} and \eqref{eq: sub heat eq 2} are equivalent
provided $z\in H^1(0,1)$ and
$\beta_{j}=-v_{j}=-\left[z_{x}\right]_{x_j}$.
 Let $\alpha_j=-u_j=\beta_j=-v_j=-[z^j_x]_{x_j}$ for any
$j=1,2,...,m$, where $z^j$ denotes the solution to System \eqref{eq:
sub heat eq 2}. One may directly check that
$z(x,t)=\sum_{j=1}^{m}z^j(x,t)$ is a solution to System \eqref{eq:
heat eq_2}. Moreover,
\begin{align}
[z_x]_{x_i}=\sum_{j=1}^{m}[z^{j}_{x}]_{x_{i}}=[z^i_x]_{x_i}=u_i,\notag
\end{align}
for any $i=1,2,...,m$. Hence $z(x,t)=\sum_{j=1}^{m}z^j(x,t)$ is a
solution to System \eqref{eq: heat eq}. Therefore, throughout this
paper, we assume $\alpha_j=-u_j=\beta_j=-v_j=-[z^j_x]_{x_j}$ for any
$j=1,2,...,m$. Due to the equivalences of System \eqref{eq: heat eq}
and \eqref{eq: heat eq_2}, and System \eqref{eq: sub heat eq} and
\eqref{eq: sub heat eq 2}, we may consider \eqref{eq: heat eq_2} and
System \eqref{eq: sub heat eq} in the following parts.

%%%%%%%%%%%%%%%%%%%%%%%%%%%%%%%%%%%%%%%%%%%%%%%%%%%%%%%%%%%%%%%%%%%%%%%%%%%%%%%
\section{Control Design Based on Zero-Dynamics Inverse and Differential Flatness}\label{Sec: Control design}
In the framework of zero-dynamics inverse, the in-domain control % signal of the the system \eqref{eq: heat eq}
is derived from the so-called forced zero-dynamics that are constructed from the original system dynamics
by replacing the control constraints by the regulation constraints. To work with the \emph{parallel connected}
system~\eqref{eq: sub heat eq}, we first split the reference signal as:
\begin{equation}\label{eq: splt ref}
  z^{D}(x,t)=\sum^m_{j=1}\gamma_{j}(x,x_j)z^d_j(x_j,t),
\end{equation}
in which the function $\gamma_{j}(x,x_j)$ will be determined in Proposition~\ref{Prop: 4} (see Section~\ref{Sec: Motion Planning}).
Denoting by $\varepsilon^{j}(t)=z^j(x_j,t)-z^{d}_{j}(x_{j},t)$ the regulation error corresponding to System~\eqref{eq: sub heat eq}
and replacing the control constraint by $\varepsilon^{j}(t)= 0$, we obtain the corresponding zero-dynamics for a
fixed $j$:
\begin{subequations}\label{eq: heat zd}
\begin{align}
  &\xi_{t}(x,t)=\xi_{xx}(x,t),\ x\in (0,x_j)\cup (x_j,1),\ t>0, \label{eq: heat zd_dynm}\\
  &\xi(x,0)=0, \label{eq: heat zd_BC} \\
  &\xi_{x}(0,t)-k_{0}\xi(0,t)=0,\label{eq: heat zd_B0} \\
  &\xi_{x}(1,t)+k_{1}\xi(1,t)=0,\label{eq: heat zd_B1} \\
  &\xi(x_{j},t)=z^{d}_{j}(x_{j},t). \label{eq: heat zd_reg}
\end{align}
\end{subequations}
For simplicity, we denote by $z^j$ and $\xi^j$ the solutions
to the $j^{th}$ systems \eqref{eq: sub heat eq} and \eqref{eq: heat
zd}, respectively. Also, we write henceforth $z^{d}_{j}(t)=z^{d}_{j}(x_j,t)$ as the reference
trajectory in the $j^{th}$ system \eqref{eq: heat zd} if there is no ambiguity. The in-domain
control signal of System \eqref{eq: sub heat eq} can then be computed by
\begin{equation}\label{eq: in-domain control}
  v_{j} = [z^j_x]_{x_j}=[\xi^j_x]_{x_j}.
\end{equation}
\begin{remark}
Note that for $k_{0}\geq 0, k_{1}\geq 0$ with $k_{0}+k_{1}\neq
0$, arguing as \cite{BGIS206}, we have
\begin{equation*}
 \varepsilon^{j}(t)=z^j(x_j,t)-z^{d}_{j}(x_j,t) \rightarrow 0 \text{ as } t\rightarrow \infty.
\end{equation*}
\end{remark}

 Obviously, to find the control signals, we need to solve the
corresponding zero-dynamics \eqref{eq: heat zd}. For this purpose, we leverage the
technique of flat systems
\cite{Fliess:1999-CDC,Meurer:2013,Rudolph2003}. In particular, we
apply a standard procedure of Laplace transform-based method to find
the solution to \eqref{eq: heat zd}. Henceforth, we denote by
$\widehat{f}(x,s)$ the Laplace transform of a function $f(x,t)$ with
respect to time variable. Then, for fixed $x_{j}\in (0,1)$, the
transformed equations of \eqref{eq: heat zd} in the Laplace domain read as
\begin{subequations}\label{eq: heat zd_laplace}
\begin{align}
  &s\widehat{\xi}(x,s)=\widehat{\xi}_{xx}(x,s),\ x\in (0,x_j)\cup (x_j,1),\ s\in\mathbb{C}, \label{eq: heat zd_laplace_dynm}\\
  &\widehat{\xi}(x,0)=0, \label{eq: heat zd_laplace_BC} \\
  &\widehat{\xi}_{x}(0,s)-k_{0}\widehat{\xi}(0,s)=0, \label{eq: heat zd_laplace_B0} \\
  &\widehat{\xi}_{x}(1,s)+k_{1}\widehat{\xi}(1,s)=0,\label{eq: heat zd_laplace_B1} \\
  &\widehat{\xi}(x_{j},s)=\widehat{z}^{d}_{j}(x_{j},s). \label{eq: heat zd_laplace_reg}
\end{align}
\end{subequations}

 We divide \eqref{eq: heat zd_laplace} into two sub-systems, i.e.,
for fixed $x_{j}\in (0,1)$, considering
\begin{subequations}\label{eq: heat zd_laplaceL}
\begin{align}
  &s\widehat{\xi}(x,s)=\widehat{\xi}_{xx}(x,s),\ 0< x< x_{j},\ s\in\mathbb{C}, \label{eq: heat zd_laplaceL_dynm}\\
  &\widehat{\xi}(x,0)=0, \label{eq: heat zdzd_laplaceL_BC} \\
  &\widehat{\xi}_{x}(0,s)-k_{0}\widehat{\xi}(0,s)=0, \label{eq: heat zd_laplaceL_B0} \\
  &\widehat{\xi}(x_{j},s)=\widehat{z}^{d}_{j}(x_{j},s), \label{eq: heat zd_laplaceL_reg}
\end{align}
\end{subequations}
and
\begin{subequations}\label{eq: heat zd_laplaceR}
\begin{align}
  &s\widehat{\xi}(x,s)=\widehat{\xi}_{xx}(x,s),\ x_{j}< x< 1,\ s\in\mathbb{C}, \label{eq: heat zd_laplaceR_dynm}\\
  &\widehat{\xi}(x,0)=0, \label{eq: heat zdzd_laplaceL_BC} \\
  &\widehat{\xi}_{x}(1,s)+k_{1}\widehat{\xi}(x,s)=0, \label{eq: heat zd_laplaceR_B1} \\
  &\widehat{\xi}(x_{j},s)=\widehat{z}^{d}_{j}(x_{j},s), \label{eq: heat zd_laplaceR_reg}
\end{align}
\end{subequations}

Let $\widehat{\xi}^{j}_{-}(x,s)$ and $\widehat{\xi}^{j}_{+}(x,s)$ be
the general solutions to \eqref{eq: heat zd_laplaceL} and \eqref{eq:
heat zd_laplaceR}, respectively, and denote their inverse Laplace
transforms by ${\xi}^{j}_{-}(x,t)$ and ${\xi}^{j}_{+}(x,t)$. The
solution to \eqref{eq: heat zd} can be written as
\begin{align*}
\xi^{j}(x,t)={\xi}^{j}_{-}(x,t)\chi_{\{(0,x_j)\}}+{\xi}^{j}_{+}(x,t)\chi_{\{[x_j,1)\}},
\end{align*}
where
\begin{equation*}
  \chi(x)_{\{\Omega_j\}}
 = \left\{
     \begin{array}{ll}
       1, & \hbox{$x \in \Omega_j\subseteq (0,1)$;} \\
       0, & \hbox{otherwise.}
     \end{array}
   \right.
\end{equation*}
Then at each point $x_{i}\in (0,1)$, by \eqref{eq: in-domain
control} and the argument of ``\emph{parallel connection}" (see
Section~\ref{Sec: Problem Setting}), we have $[z_x]_{x_i}
 =\sum_{j=1}^{m}[z^j_x]_{x_i}=
[z^i_x]_{x_i}=[\xi^i_x]_{x_i}, \ i = 1,\ldots, m$. Hence the
in-domain control signals of System \eqref{eq: heat eq} can be
computed by
\begin{equation}\label{control u}
  u_{i} = [z_x]_{x_i} = [\xi^i_x]_{x_i}, \ i = 1,\ldots, m.
\end{equation}

In the following steps, we present the computation of the solution
to System \eqref{eq: heat zd}, $\xi^{j}$. Issues related to the
reference trajectory $z^D(x,t)$ for System \eqref{eq: heat eq} will
be addressed in Section~\ref{Sec: Motion Planning}.

Note that $\widehat{\xi}^{j}_{-}(x,s)$ and
 $\widehat{\xi}^{j}_{+}(x,s)$,
the general solutions to \eqref{eq: heat zd_laplaceL} and \eqref{eq:
heat zd_laplaceR}, are given by
\begin{align*}
\widehat{\xi}^{j}_{-}(x,s)&=C_{1}\phi_{1}(x,s)+C_{2}\phi_{2}(x,s), \\
\widehat{\xi}^{j}_{+}(x,s)&=C_{3}\phi_{1}(x,s)+C_{4}\phi_{2}(x,s),
\end{align*}
with
\begin{align*}
\phi_{1}(x,s)=\frac{\sinh(\sqrt{s}x)}{\sqrt{s}}, \ \phi_{2}(x,s)=\cosh(\sqrt{s}x).
\end{align*}
We obtain by applying \eqref{eq: heat zd_laplaceL_B0} and \eqref{eq:
heat zd_laplaceL_reg}
\begin{align*}
    &C_{1}\phi_{1}(x_{j},s)+ C_{2}\phi_{2}(x_{j},s)=\widehat{z}^{d}_{j}(x_{j},s),\  C_{1}-k_{0}C_{2}=0,
\end{align*}
which can be written as
\begin{equation*}
  \begin{pmatrix}
    \phi_{1}(x_{j},s) & \phi_{2}(x_{j},s) \\
    1  &  -k_{0}
  \end{pmatrix}
  \begin{pmatrix}
    C_{1}\\
    C_{2}\\
  \end{pmatrix}
  =
  \begin{pmatrix}
    \widehat{z}^{d}_{j}(x_{j},s)\\
    0 \\
  \end{pmatrix}.
\end{equation*}
Let
\begin{equation*}
R^{j}_{-}=\begin{pmatrix}
            \phi_{1}(x_{j},s) & \phi_{2}(x_{j},s) \\
            1  &  -k_{0}
          \end{pmatrix}
\end{equation*}
and
\begin{eqnarray}\label{w^{j}1}
\widehat{z}^{d}_{j}(x_{j},s)=-\text{det}(R^{j}_{-})\widehat{y}^{j}_{-}(x_{j},s).
\end{eqnarray}
We obtain
\begin{equation*}
 \begin{pmatrix}
    C_{1}\\
    C_{2}\\
  \end{pmatrix}
 =\dfrac{\text{adj}(R^{j}_{-})}{\text{det}(R^{j}_{-})}
   \begin{pmatrix}
    \widehat{z}^{d}_{j}(x_{j},s)\\
    0 \\
   \end{pmatrix}
 =
  \begin{pmatrix}
    k_{0}\widehat{y}^{j}_{-}(x_{j},s)\\
    \widehat{y}^{j}_{-}(x_{j},s) \\
  \end{pmatrix}.
\end{equation*}
Therefore, the solution to \eqref{eq: heat zd_laplaceL} can be expressed as
\begin{eqnarray}\label{eq: heat zd_laplaceL_sol}
\widehat{\xi}^{j}_{-}(x,s)=(k_{0}\phi_{1}(x)+\phi_{2}(x))\widehat{y}^{j}_{-}(x_{j},s).
\end{eqnarray}

 We may proceed in the same way to deal with \eqref{eq: heat zd_laplaceR}. Indeed, letting
\begin{eqnarray*}
R^{j}_{+}=
  \begin{pmatrix}
    \phi_{1}(x_{j},s) & \phi_{2}(x_{j},s) \\
    \phi_{2}(1,s)+k_{1}\phi_{1}(1,s)&s\phi_{1}(1,s)+k_{1}\phi_{2}(1,s) \\
  \end{pmatrix}
\end{eqnarray*}
and
\begin{eqnarray}\label{w^{j}2}
\widehat{z}^{d}_{j}(x_{j},s)=\text{det}(R^{j}_{+})\widehat{y}^{j}_{+}(x_{j},s),
\end{eqnarray}
we get from \eqref{eq: heat zd_laplaceR}
\begin{eqnarray*}
  \begin{pmatrix}
    C_{3}\\
    C_{4}\\
  \end{pmatrix}
  =
  \begin{pmatrix}
    (s\phi_{1}(1,s)+k_{1}\phi_{2}(1,s))\widehat{y}^{j}_{+}(x_{j},s)\\
    -(\phi_{2}(1,s)+k_{1}\phi_{1}(1,s))\widehat{y}^{j}_{+}(x_{j},s) \\
  \end{pmatrix},
\end{eqnarray*}
and
\begin{align}\label{eq: heat zd_laplaceR_sol}
   \widehat{\xi}^{j}_{+}(x,s)
 = &\left((s\phi_{1}(1,s)+k_{1}\phi_{2}(1,s))\phi_{1}(x)\right. \nonumber\\
   &+\left.(\phi_{2}(1,s)+k_{1}\phi_{1}(1,s))\phi_{2}(x)\right)\widehat{y}^{j}_{+}(x_{j},s).
\end{align}

Applying modulus theory \cite{Mounier:1995,Rouchon:2001} to \eqref{w^{j}1} and \eqref{w^{j}2}, we may
choose $\widehat{y}_{j}(x_{j},s)$ as the basic output such that
\begin{align}
\widehat{y}^{j}_{+}(x_{j},s)&=-\text{det}(R^{j}_{-})\widehat{y}_{j}(x_{j},s),\label{eq: y+}\\
\widehat{y}^{j}_{-}(x_{j},s)&=\text{det}(R^{j}_{+})\widehat{y}_{j}(x_{j},s). \label{eq: y-}
\end{align}
Then, using the property of hyperbolic functions, we obtain from
\eqref{eq: heat zd_laplaceL_sol} and \eqref{eq: heat zd_laplaceR_sol} that
\begin{align}
\widehat{\xi}^{j}_{-}(x,s)=&\left(k_{1}\frac{\sinh(\sqrt{s}x_{j}-\sqrt{s})}{\sqrt{s}}-\cosh(\sqrt{s}x_{j}-\sqrt{s})\right) \nonumber\\
                           &\times \left(k_{0}\frac{\sinh(\sqrt{s}x)}{\sqrt{s}} +\cosh(\sqrt{s}x)\right)\widehat{y}_{j}(x_{j},s),\\
\widehat{\xi}^{j}_{+}(x,s)=&\left(k_{1}\frac{\sinh(\sqrt{s}x-\sqrt{s})}{\sqrt{s}}-\cosh(\sqrt{s}x-\sqrt{s})\right) \nonumber\\
                           &\times \left(k_{0}\frac{\sinh(\sqrt{s}x_{j})}{\sqrt{s}}+\cosh(\sqrt{s}x_{j})\right)\widehat{y}_{j}(x_{j},s).
\end{align}

Note that
\begin{align}\label{eq: solution in s-domain}
\widehat{\xi}^{j}(x,s)=\widehat{\xi}^{j}_{-}(x,s)\chi_{\{(0,x_{j})\}}+\widehat{\xi}^{j}_{+}(x,s)\chi_{\{[x_{j},1)\}}
\end{align}
is a solution to \eqref{eq: heat zd_laplace}. Using the fact
\begin{eqnarray*}
\sinh x=\sum^{\infty}_{n=0}\frac{x^{2n+1}}{(2n+1)!}, \
\cosh x=\sum^{\infty}_{n=0}\frac{x^{2n}}{(2n)!},
\end{eqnarray*}
we obtain
\begin{align}
  &\widehat{\xi}^{j}(x,s) \notag\\
 =&\Bigg[\Bigg(k_{0}k_{1}\sum_{n=0}^{\infty}\sum_{k=0}^{n}\frac{x^{2k+1}(x_{j}-1)^{2(n-k)+1}}{(2k+1)![2(n-k)+1]!}s^{n}
   -k_{0}\sum_{n=0}^{\infty}\sum_{k=0}^{n}\frac{x^{2k+1}(x_{j}-1)^{2(n-k)}}{(2k+1)![2(n-k)]!}s^{n}\notag\\
  &+k_{1}\sum_{n=0}^{\infty}\sum_{k=0}^{n}\frac{x^{2k}(x_{j}-1)^{2(n-k)+1}}{(2k)![2(n-k)+1]!}s^{n}
   -\sum_{n=0}^{\infty}\sum_{k=0}^{n}\frac{x^{2k}(x_{j}-1)^{2(n-k)}}{(2k)![2(n-k)]!}s^{n}\Bigg)\chi_{\{(0,x_{j})\}}\notag\\
  &+\Bigg(k_{0}k_{1}\sum_{n=0}^{\infty}\sum_{k=0}^{n}\frac{x_{j}^{2k+1}(x-1)^{2(n-k)+1}}{(2k+1)![2(n-k)+1]!}s^{n}
   -k_{0}\sum_{n=0}^{\infty}\sum_{k=0}^{n}\frac{x_{j}^{2k+1}(x-1)^{2(n-k)}}{(2k+1)![2(n-k)]!}s^{n}\notag\\
  &+ k_{1}\sum_{n=0}^{\infty}\sum_{k=0}^{n}\frac{x_{j}^{2k}(x-1)^{2(n-k)+1}}{(2k)![2(n-k)+1]!}s^{n}
   -\sum_{n=0}^{\infty}\sum_{k=0}^{n}\frac{x_{j}^{2k}(x-1)^{2(n-k)}}{(2k)![2(n-k)]!}s^{n}\Bigg)\chi_{\{[x_{j},1)\}}\Bigg]\widehat{y}_{j}.
    \notag
\end{align}
It follows that
\begin{align}\label{eq: xi_j}
  &\xi^{j}(x,t) \notag\\
 =&\Bigg[\Bigg(k_{0}k_{1}\sum_{n=0}^{\infty}\sum_{k=0}^{n}\frac{x^{2k+1}(x_{j}-1)^{2(n-k)+1}}{(2k+1)![2(n-k)+1]!}y_{j}^{(n)}
   -k_{0}\sum_{n=0}^{\infty}\sum_{k=0}^{n}\frac{x^{2k+1}(x_{j}-1)^{2(n-k)}}{(2k+1)![2(n-k)]!}y_{j}^{(n)}\notag\\
  &+k_{1}\sum_{n=0}^{\infty}\sum_{k=0}^{n}\frac{x^{2k}(x_{j}-1)^{2(n-k)+1}}{(2k)![2(n-k)+1]!}y_{j}^{(n)}
   -\sum_{n=0}^{\infty}\sum_{k=0}^{n}\frac{x^{2k}(x_{j}-1)^{2(n-k)}}{(2k)![2(n-k)]!}y_{j}^{(n)}\Bigg)\chi_{\{(0,x_{j})\}}\notag\\
  &+\Bigg(k_{0}k_{1}\sum_{n=0}^{\infty}\sum_{k=0}^{n}\frac{x_{j}^{2k+1}(x-1)^{2(n-k)+1}}{(2k+1)![2(n-k)+1]!}y_{j}^{(n)}
   -k_{0}\sum_{n=0}^{\infty}\sum_{k=0}^{n}\frac{x_{j}^{2k+1}(x-1)^{2(n-k)}}{(2k+1)![2(n-k)]!}y_{j}^{(n)}\notag\\
  &+ k_{1}\sum_{n=0}^{\infty}\sum_{k=0}^{n}\frac{x_{j}^{2k}(x-1)^{2(n-k)+1}}{(2k)![2(n-k)+1]!}y_{j}^{(n)}
   -\sum_{n=0}^{\infty}\sum_{k=0}^{n}\frac{x_{j}^{2k}(x-1)^{2(n-k)}}{(2k)![2(n-k)]!}y_{j}^{(n)}\Bigg)\chi_{\{[x_{j},1)\}}\Bigg].
\end{align}
By a direct computation we get
\begin{align}
\left[\widehat{\xi}^{j}_{x}\right]_{x_{j}}
  =&\left[\left(\frac{k_{0}k_{1}}{\sqrt{s}}+\sqrt{s}\right)\sinh(\sqrt{s})+(k_{0}+k_{1})\cosh(\sqrt{s})\right]
    \times \widehat{y}_{j}(x_{j},s)\nonumber\\
  =&\left(k_{0}k_{1}\sum^{\infty}_{n=0}\frac{s^{n}}{(2n+1)!}+(k_{0}+k_{1})\sum^{\infty}_{n=0}\frac{s^{n}}{(2n)!}\right.
    \left. +\sum^{\infty}_{n=0}\frac{s^{n+1}}{(2n+1)!}\right)\widehat{y}_{j}(x_{j},s).
\end{align}
It follows from \eqref{control u} that
\begin{align}\label{eq: control u_j}
u_{j}(t)=&\left[\xi^{j}_{x}\right]_{x_{j}} \nonumber\\
        =&k_{0}k_{1}\sum^{\infty}_{n=0}\frac{y^{(n)}_{j}(x_{j},t)}{(2n+1)!}+(k_{0}+k_{1})\sum^{\infty}_{n=0}\frac{y^{(n)}_{j}(x_{j},t)}{(2n)!}
          +\sum^{\infty}_{n=0}\frac{y^{(n+1)}_{j}(x_{j},t)}{(2n+1)!}.
\end{align}

\begin{remark}
The reference signal $z^d_j(x_j,t)$ can be derived in the same way from the flat output from \eqref{w^{j}1} and \eqref{w^{j}2}. However, as
the flatness-based control is driven by flat output, there is no need to explicitly compute $z^d_j(x_j,t)$.
\end{remark}

%%%%%%%%%%%%%%%%%%%%%%%%%%%%%%%%%%%%%%%%%%%%%%%%%%%%%%%%%%%%%%%%%%%%%%%
\section{Motion Planning}\label{Sec: Motion Planning}
For control purpose, we have to choose appreciate reference trajectories, or equivalently the basic outputs.
In the present work, we consider the set-point control problem, i.e. to steer the heat distribution
to a desired steady-state profile, denoted by $\bar{z}^D(x)$.
Without loss of generality, we consider a set of basic outputs of the form:
\begin{equation}\label{eq: ref traj}
  y_j(t) = \overline{y}(x_j)\varphi_j(t), \ j = 1,\ldots, m,
\end{equation}
where $\varphi_j(t)$ is a smooth function evolving from 0 to 1. Motion planning amounts then to deriving
$\overline{y}(x_j)$ from $\bar{z}^D(x)$ and to determining appropriate functions $\varphi_j(t)$, for $j = 1, \ldots, m$.

To this aim and due to the equivalence of the systems \eqref{eq: heat eq} and \eqref{eq: heat eq_2}, we
consider the steady-state heat equation corresponding to System~\eqref{eq: heat eq_2}:
\begin{subequations}\label{eq: heat_eq_ss}
\begin{align}
  &\bar{z}_{xx}(x)=\sum_{j=1}^{m}\delta (x-x_{j})\bar{\alpha}_{j},\ 0< x<1,\ t>0, \label{eq: heat eq_ss_dynm}\\
  &\bar{z}_{x}(0)-k_{0}\bar{z}(0)=0,\label{eq: heat eq_ss_B0} \\
  &\bar{z}_{x}(1)+k_{1}\bar{z}(1)=0.\label{eq: heat eq_ss_B1}
\end{align}
\end{subequations}

Based on the principle of superposition for linear systems, the solution to the steady-state heat
equation \eqref{eq: heat_eq_ss} can be expressed as:
\begin{align}\label{+14}
  \bar{z}(x)=\int_{0}^1\sum_{j=1}^{m}G(x,\zeta)\delta
  (\zeta-x_{j})\bar{\alpha}_{j}\text{d}\zeta=\sum_{j=1}^{m}G(x,x_{j})\bar{\alpha}_{j}.
\end{align}
where $G(x,\zeta)$ is the Green's function corresponding to \eqref{eq: heat_eq_ss}, which is of the form
\begin{eqnarray}\label{eq: Geen}
G(x,\zeta)=\left\{
  \begin{array}{ll}
    \dfrac{(k_{1}\zeta-k_{1}-1)(k_{0}x+1)}{k_{0}+k_{1}+k_{0}k_{1}},\  0\leq x<\zeta;~~~ \hbox{} \\
    \dfrac{(k_{1}x-k_{1}-1)(k_{0}\zeta+1)}{k_{0}+k_{1}+k_{0}k_{1}}, \ \zeta\leq x\leq 1.\hbox{}
  \end{array}
\right.
\end{eqnarray}
Indeed, it is easy to check that $G_{xx}(x,\zeta)=\delta(x-\zeta)$
and $G(x,\zeta)$ satisfies the boundary conditions, $G_{x}(0,\zeta)-k_{0}G(0,\zeta)=0$
and $G_{x}(1,\zeta)+k_{1}G(1,\zeta)=0$,
the joint condition, $G(\zeta^+,\zeta)= G(\zeta^-,\zeta)$,
and the jump condition, $[G_{x}(x,\zeta)]_{\zeta}=1$.

 Taking $m$ distinguished points along
the solution to \eqref{eq: heat_eq_ss},
$\bar{z}(x_{1},t),\ldots,\bar{z}(x_{m},t)$, we get
\begin{align}\label{eq: map}
  \begin{pmatrix}
    \bar{z}(x_{1}) \\
    \vdots \\
    \bar{z}(x_{m})
  \end{pmatrix}
 =
  \begin{pmatrix}G(x_{1},x_{1}) & \cdots& G(x_{1},x_{m})\\
  \vdots & \ddots&\vdots\\
   G(x_{m},x_{1}) &\cdots&G(x_{m},x_{m})\\
  \end{pmatrix}
  \begin{pmatrix}
    \bar{\alpha}_{1} \\
    \vdots \\
    \bar{\alpha}_{m}
  \end{pmatrix}.
\end{align}
\ \\
\begin{mylemma}\label{Prop: 2}
The matrix $(G(x_{i},x_{j}))_{m\times m}$ chosen as in \eqref{eq: map}
is invertible. Thus
\begin{align}\label{eq: inv map}
  \begin{pmatrix}
    \bar{\alpha}_{1} \\
    \vdots \\
    \bar{\alpha}_{m}
  \end{pmatrix}
 =
  \begin{pmatrix}G(x_{1},x_{1}) & \cdots& G(x_{1},x_{m})\\
  \vdots & \ddots&\vdots\\
   G(x_{m},x_{1}) &\cdots&G(x_{m},x_{m})\\
  \end{pmatrix}^{-1}
  \begin{pmatrix}
    \bar{z}(x_{1}) \\
    \vdots \\
    \bar{z}(x_{m})
  \end{pmatrix}.
\end{align}
\end{mylemma}

\begin{proof}
For $m=1$, since $k_0\geq 0,k_{1}\geq 0, k_0+k_1>0$ and $x_{1}\in
(0,1)$, it follows that
$G(x_{1},x_{1})=\frac{(k_{1}x_{1}-k_{1}-1)(k_{0}x_{1}+1)}{k_{0}+k_{1}+k_{0}k_{1}}<0$.
Hence it is invertible.  We prove the claim for $m>1$ by contradiction.
Suppose that the matrix $(G(x_{i},x_{j}))_{m\times m}$ is not
invertible, then it is of rank $m-1$ or less. Without loss of
generality, we may assume that for some $x_n>x_i$ with
$i=1,...,n-1$, there exist $n-1$ constants $l_{1},\ l_{2}, ..., \
l_{n-1}$ such that
\begin{subequations}\label{green function at points}
\begin{align}
G(x_{1},x_{n})&=\sum_{i=1}^{n-1}l_{i}G(x_{1},x_{i}),\\
%G(x_{2},x_{n})&=\sum_{i=1}^{n-1}l_{i}G(x_{2},x_{i}),\\
&\vdots\notag\\
G(x_{n},x_{n})&=\sum_{i=1}^{n-1}l_{i}G(x_{n},x_{i}),
\end{align}
\end{subequations} where $1<n\leq m$ and $\sum_{i=1}^{n-1}l^2_i>0$. Let
\begin{align*}
G(x)=G(x,x_{n}),\ \ F(x)=\sum_{i=1}^{n-1}l_{i}G(x,x_{i}).
\end{align*}
\eqref{green function at points} shows that $F(x)=G(x)$ at every
boundary point of $[x_1,x_2],\ [x_2,x_3], \ldots,[x_{n-1},x_n]$. Note
that $F(x)$ is a linear function in $[x_1,x_2],\ [x_2,x_3], \ldots,[x_{n-1},x_n]$, and that
$G(x)=\frac{(k_{1}x_n-k_{1}-1)(k_{0}x+1)}{k_{0}+k_{1}+k_{0}k_{1}}$
in $[x_1,x_n]$, i.e., $G(x)$ is a linear function in $[x_1,x_n]$.
Hence $F(x)\equiv G(x)$ in $[x_1,x_n]$.

By $F(x_1)=G(x_1)$, we get
\begin{align}\label{F(x1)=G(x1)}
k_{1} x_{n}-k_{1}-1=\sum^{n-1}_{i=1}l_{i}(k_{1}x_{i}-k_{1}-1).
\end{align}
By $F(x_n)=G(x_n)$, we get
\begin{align}\label{F(xn)=G(xn)}
k_{0} x_{n}+1=\sum^{n-1}_{i=1}l_{i}(k_{0}x_{i}+1).
\end{align}
Therefore
\begin{align}\label{sum li =1}
 \sum^{n-1}_{i=1}l_{i}=  1.
\end{align}
By $F_{x}(x_1^+)=G_{x}(x_1^+)$ and $F_{x}(x_n^-)=G_{x}(x_n^-)$, we get
\begin{align}
  k_0(k_{1} x_{n}-k_{1}-1)
 =&k_0k_1\sum^{n-1}_{i=1}l_{i}x_i- k_0(k_1+1)\sum^{n-1}_{i=2}l_{i}+l_1k_1,\nonumber\\
 =&k_0k_1\sum^{n-1}_{i=1}l_{i}x_i+k_1\sum^{n-1}_{i=1}l_{i}. \label{DF(xn)=DG(xn)}
\end{align}
It follows that $\sum^{n-1}_{i=2}l_{i}=0$,
which yields, considering \eqref{sum li =1}, $l_{1}=1$.
By $F_{x}(x_2^+)=G_{x}(x_2)=F_{x}(x_2^-)$, we deduce
\begin{align*}
 &l_1k_1(k_0x_1+1)+l_2k_1(k_0x_2+1)+k_0\sum^{n-1}_{i=3}l_{i}(k_1x_i-k_1-1)\\
=&l_1k_1(k_0x_1+1)+k_0\sum^{n-1}_{i=2}l_{i}(k_1x_i-k_1-1),\notag
\end{align*}
which gives $l_{2}=0$.
Similarly, by $F_{x}(x_j^+)=G_{x}(x_j)=F_{x}(x_j^-) \
(j=3,4,...,n-2)$, we obtain
$%\begin{align*}
l_{3}=l_4=...=l_{n-2}=0.
$ %\end{align*}
Hence $l_{n-1}=0$. Then we deduce from \eqref{DF(xn)=DG(xn)} that
\begin{align}\label{contradiction 1}
k_0(k_1x_n-k_1-1)=k_1(k_0x_1+1).
\end{align}
It follows from \eqref{F(xn)=G(xn)} that
\begin{align}\label{contradiction 2}
k_0x_1+1=k_0x_n+1.
\end{align}
We conclude by \eqref{contradiction 1} and \eqref{contradiction 2}
that $k_0+k_1+k_0k_1=0$, which is a contradiction to $k_0\geq 0,\
k_1\geq 0$, and $k_0+k_1>0$. \hfill $\Box$
\end{proof}

 In steady-state, we can obtain from \eqref{eq: control u_j} that
\begin{eqnarray}\label{eq: static input-output}
\overline{u}_{j}= (k_{0}k_{1}+k_{0}+k_{1})\overline{y}(x_{j}) = -\bar{\alpha}_{j}.
\end{eqnarray}
Finally, $\overline{y}(x_{j})$ can be computed by \eqref{eq: inv map} and
\eqref{eq: static input-output} for a given $\bar{z}^D(x)$.

It is worth noting that \eqref{eq: inv map} provides a simple and straightforward way to compute the static control
from the prescribed steady-state profile. Indeed, a direct computation can show that applying \eqref{eq: inv map} will result in the
same static control obtained in \cite{BGIS206} where a \emph{serially connected} model is used.

To ensure the convergence of \eqref{eq: xi_j} and \eqref{eq: control u_j}, we
choose the following smooth function as $\varphi_j(t)$:
\begin{equation}\label{eq: Gevrey function}
\varphi_j(t) = \left\{
\begin{array}{l l}
  0, & \quad \mbox{if $t \leq 0$}\\
  \dfrac{\displaystyle\int_0^t \exp(-1/(\tau(1-\tau)))^{\varepsilon}d\tau}{\displaystyle\int_0^T \exp(-1/(\tau(1-\tau)))^{\varepsilon}d\tau}, & \quad \mbox{if $t\in (0,T)$}\\
  1, & \quad \mbox{if $t \geq T$}\\
  \end{array} \right.
\end{equation}
which is known as Gevrey function of order $\sigma=1+1/\varepsilon$,
$\varepsilon>0$ (see, e.g., \cite{Rudolph2003}).

\begin{mylemma}\label{Prop: 3}
If the basic outputs $\varphi_j(t)$, $j = 1, \ldots, m$, are chosen
as Gevrey functions of order $1<\sigma<2$, then the infinite series
\eqref{eq: xi_j} and \eqref{eq: control u_j} are convergent.
\end{mylemma}
\begin{proof}
We prove the convergence of the power series \eqref{eq: xi_j}
and \eqref{eq: control u_j} using Cauchy-Hadamard Theorem. Indeed, it suffices
to prove the convergence of
\begin{equation}\label{+21}
\sum_{n=0}^{\infty}\sum_{k=0}^{n}\frac{1}{(2k)![2(n-k)!]}\varphi^{(n)}_{j}(t).
\end{equation}
Then the convergence of the series in \eqref{eq: xi_j} and
\eqref{eq: control u_j} follows easily using the same argument.
%
%We use the convergence theorem of Cauchy-Hadamard. By that, a series $\sum_{k}^{\infty}b_k$ is absolutely convergent if $\limsup_{k\rightarrow\infty}\sqrt[k]{|b_k|}<1$.

We recall that the bounds of Gevrey functions of order $\sigma$ are
given by \cite{DPM03}
\begin{equation}\label{eq: gevrey bounds}
    \exists K, M>0, \forall k \in \mathbb{Z}_{\geq 0}, \forall t\in[t_0,T], \left| \varphi^{(k+1)}(t)\right|\leq M\frac{(k!)^{\sigma}}{K^k}.
\end{equation}
Denote in \eqref{+21}
\begin{equation*}
    b_n= \sum_{k=0}^{n}\frac{1}{(2k)![2(n-k)!]}\varphi^{(n)}_{j}(t).
\end{equation*}
Then, \eqref{+21} converges if
$\limsup_{n\rightarrow\infty}\sqrt[n]{|b_n|}<1$.
Now $b_n$ can be estimated by \eqref{eq: gevrey bounds}
\begin{align*}
    |b_{n}| & \leq \sum_{k=0}^{n}\frac{M}{(2k)![2(n-k)!]}\frac{(n!)^{\sigma}}{K^n} \notag\\
            & \leq M\frac{2^n}{(2n)!}\frac{(n!)^{\sigma}}{K^n}.\notag
\end{align*}
Therefore
\begin{align*}
\limsup_{n\rightarrow\infty}\sqrt[n]{|b_n|} &\leq
\limsup_{n\rightarrow\infty}\frac{2}{K} M^{1/n}\frac{
\left((n!)^{1/n}\right)^{\sigma}}{{\left(((2n)!)^{1/2n}\right)^{2}}} \nonumber\\
&\leq\limsup_{n\rightarrow\infty}\frac{2}{K}
\frac{(n/e)^\sigma}{(2n/e)^2}\notag\\
&=\frac{e^{2-\sigma}}{2K}\limsup_{n\rightarrow\infty}
n^{\sigma-2}=\left\{
   \begin{array}{ll}
     0, & \hbox{$\sigma<2$,} \\
     \dfrac{1}{2K}, & \hbox{$\sigma=2$,} \\
      \infty, & \hbox{$\sigma>2$,}
   \end{array}
 \right.
\end{align*}
where in the second inequality we applied Stirling's formaula
$\sqrt[n]{n!} \simeq (n/e)$. We can conclude by Cauchy-Hadamard
Theorem that \eqref{+21} converges for $\sigma<2$, and for
$\sigma=2$ if $2K>1$. The series \eqref{+21} diverges if $\sigma>2$.
\hfill $\Box$
\end{proof}

\begin{mytheorem}\label{Prop: 4}
Assume $k_0\geq 0, \ k_1\geq 0,$ with $k_0+k_1>0$. Let the basic
outputs $\varphi_j(t)$, $j = 1, \ldots, m$, be chosen as \eqref{eq:
Gevrey function} with an order $1<\sigma<2$. Let the reference trajectory
of System \eqref{eq: heat eq} be given by \eqref{eq: splt ref} with
\begin{equation}\label{eq: gama}
  \gamma_{j}(x,x_j) =-\dfrac{(k_0k_1+k_0+k_1)G(x,x_j)}{(k_0x_j+1)(k_0(x_j-1)-1)}, \ j = 1,\ldots, m,
\end{equation}
where
$G(x,\zeta)$ is the Green's function defined by \eqref{eq: Geen}. Then the regulation error of System~\eqref{eq: heat eq} with the control given in \eqref{eq: control u_j} tends to zero, i.e.,
\begin{align*}
e_{i}(t)=z(x_{i},t)-z^{D}_{i}(x_i,t)\rightarrow 0\ \ \text{as}\
t\rightarrow \infty,
\end{align*}
for $i=1,2,...,m.$
\end{mytheorem}

\begin{proof}
By a direct computation we have
\begin{align*}
 |e_i(t)|
=&|z(x_{i},t)-z^{D}_i(x_{i},t)|\notag\\
=& \left|z(x_{i},t)+\sum_{j=1}^m
\frac{(k_0k_1+k_0+k_1)G(x_i,x_j)z^d_j(x_j,t)}{(k_0x_j+1)(k_0(x_j-1)-1)}\right|\notag\\
=
&\left|z(x_{i},t)+\sum_{j=1}^m\frac{(k_0k_1+k_0+k_1)G(x_i,x_j)\xi^j(x_j,t)}{(k_0x_j+1)(k_0(x_j-1)-1)}\right|\notag\\
\leq &\left|z(x_{i},t)-\bar{z}(x_i)\right|
 +\left|\bar{z}(x_i)+
(k_0k_1+k_0+k_1)\sum_{j=1}^m G(x_i,x_j)\overline{y}(x_{j})\right|\notag\\
&+\left|(k_0k_1+k_0+k_1)\sum_{j=1}^m \frac{G(x_i,x_j)\xi^j(x_j,t)}{(k_0x_j+1)(k_0(x_j-1)-1)}\right.\notag\\
&\left.-(k_0k_1+k_0+k_1)\sum_{j=1}^m G(x_i,x_j)\overline{y}(x_{j})\right|\notag.
\end{align*}
By \eqref{eq: static input-output} and \eqref{+14}, it follows
\begin{align*}
\bar{z}(x_i)=- (k_0k_1+k_0+k_1)\sum_{j=1}^m
G(x_i,x_j)\overline{y}(x_{j}).
\end{align*}
Based on \eqref{eq: xi_j}, \eqref{eq: ref traj}, and the property of $\varphi_j(t)$ we have
\begin{align*}
\frac{\xi^j(x_j,t)}{(k_0x_j+1)(k_0(x_j-1)-1)}\rightarrow
\overline{y}(x_{j})\ \ \text{as}\ t\rightarrow \infty.
\end{align*}
Note that
$
z(x_{i},t)\rightarrow \bar{z}(x_i)\ \ \text{as}\ t\rightarrow
\infty.
$
Therefore
$
|e_{i}(t)|\rightarrow 0\ \ \text{as}\ t\rightarrow \infty.
$
\hfill $\Box$
\end{proof}
\begin{remark}
For any $x\in (0,1)$, replace $x_i$ by $x$ in the proof of
Proposition~\ref{Prop: 4}, we can get
$%\begin{align*}
|z(x,t)-z^{D}(x,t)|\rightarrow 0\ \ \text{as}\ t\rightarrow \infty,
$ %\end{align*}
which shows that the solution $z(x,t)$ of System \eqref{eq: heat eq}
converges to the reference trajectory $z^{D}(x,t)$ at every point
$x\in (0,1)$.
\end{remark}

%%%%%%%%%%%%%%%%%%%%%%%%%%%%%%%%%%%%%%%%%%%%%%%%%%%%%%%%%%%%%%%%%%%%%%%%%%%%%%%%%%%%%%%%%%%%%%%%%%%%%%%%%%%%%%%%%%
\section{Simulation Study} \label{simulation}
In the simulation, we implement System~(\ref{eq: heat eq_2}) with $12$ actuators evenly distributed in the domain at the spot points $\{1/13, 2/13, \ldots, 12/13\}$. The heat distribution and the time are all represented in normalized coordinates. The numerical implementation is based on a PDE solver, \textsf{pdepe}, in Matlab PDE Toolbox. 200 points in space and 50 points in time are used for the region $[0, 1]\times [0, 2]$ in numerical simulation. The basic outputs $\varphi_j(t)$ used in the simulation are Gevrey functions of the same order. In
order to meet the convergence condition given in Proposition~\ref{Prop: 3}, the order of Gevrey functions is set to $\sigma = 1.1$. The feedback boundary control gains are chosen as $k_0 = k_1 = 10$. The initial condition in simulation is set to $z(x,0)=\cos(\pi x)$.

The desired heat distribution and the evolution of heat distribution of the system controlled by the developed algorithm are depicted in Fig.~\ref{fig: Solution}. Snapshots of regulation errors are presented in Fig.~\ref{fig: errors}, which show that the regulation error tends to 0 along the space. The control signals that steer the heat distribute from the initial profile to track the de sired one are illustrated in Fig.~\ref{fig: Controls}. The simulation results show that the system performs very well with affordable control efforts.

%%%%%%%%%%%%%%%%%%%%%%%%%%%%%%%%%%%%%%%%%%%%%
\begin{figure}[thpb]\
  \centering
  \subfigure[]{\label{fig: profile}  \includegraphics[scale=.42]{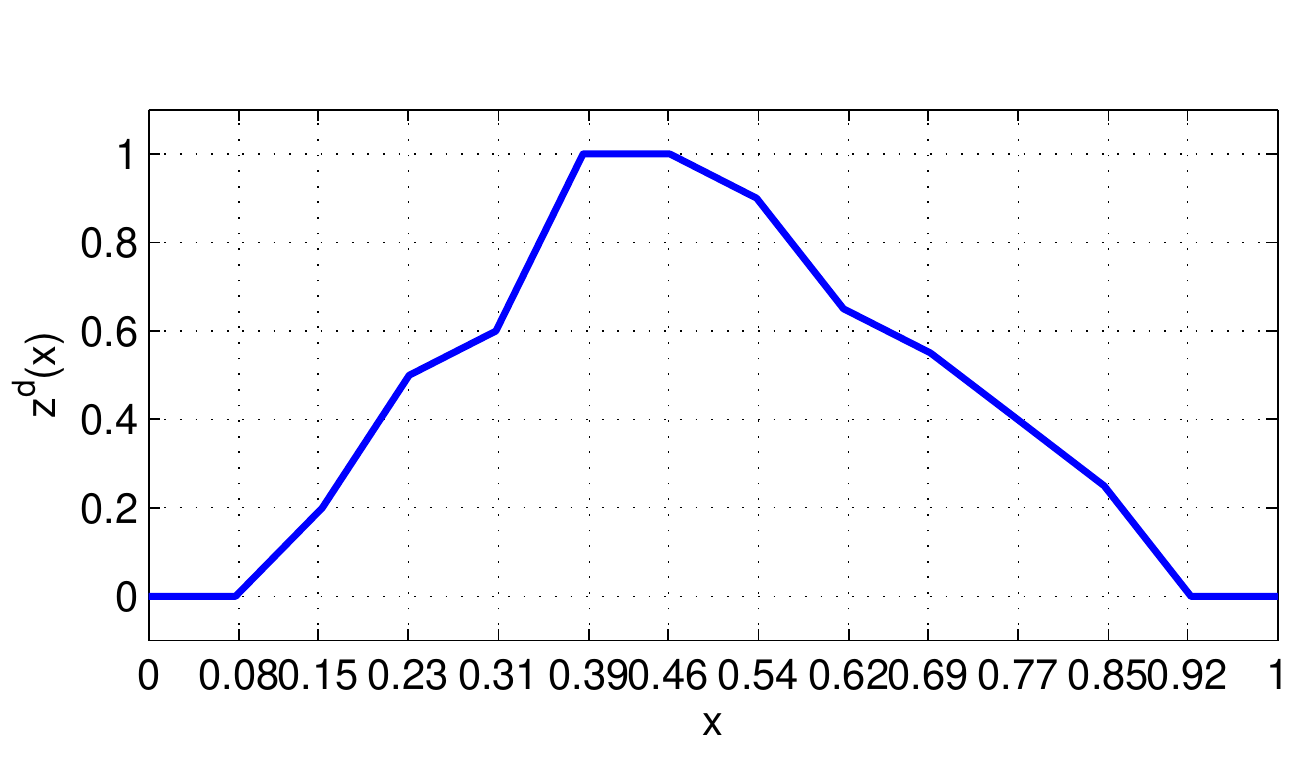}}
  \subfigure[]{\label{fig: sol surf} \includegraphics[scale=.42]{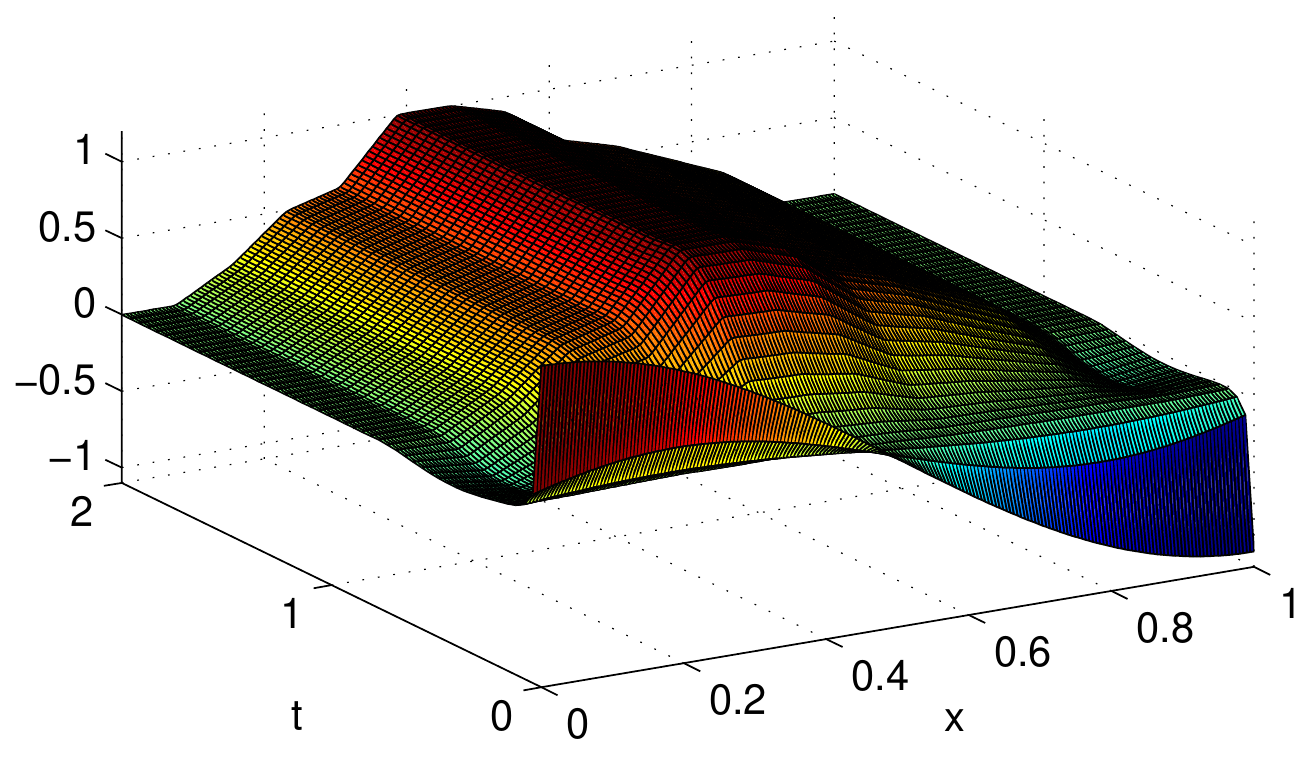}}
  \caption{Evolution of heat distribution:
           (a) desired profile;
           (b) solution surface.}
  \label{fig: Solution}
\end{figure}

%%%%%%%%%%%%%%%%%%%%%%%%%%%%%%%%%%%%%%%%%%%%%%%%%%
\begin{figure}[thpb]\
  \centering
  \subfigure[]{\label{fig: errors_2}\includegraphics[scale=.45]{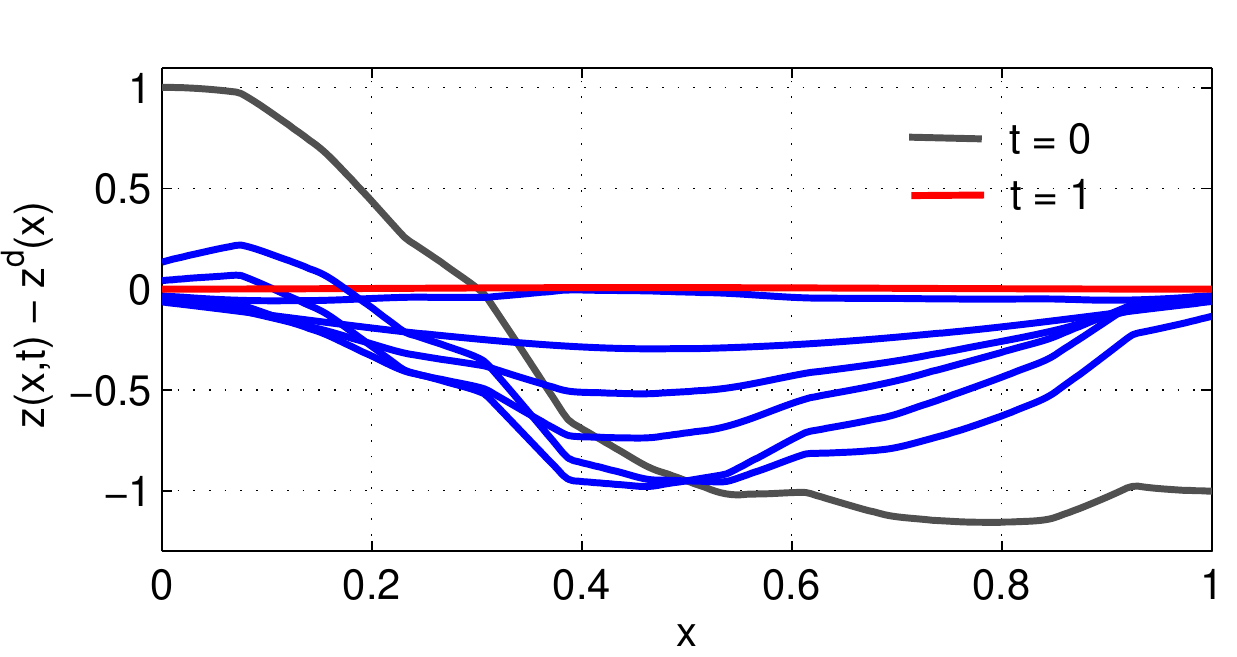}}
  \subfigure[]{\label{fig: errors_2}\includegraphics[scale=.45]{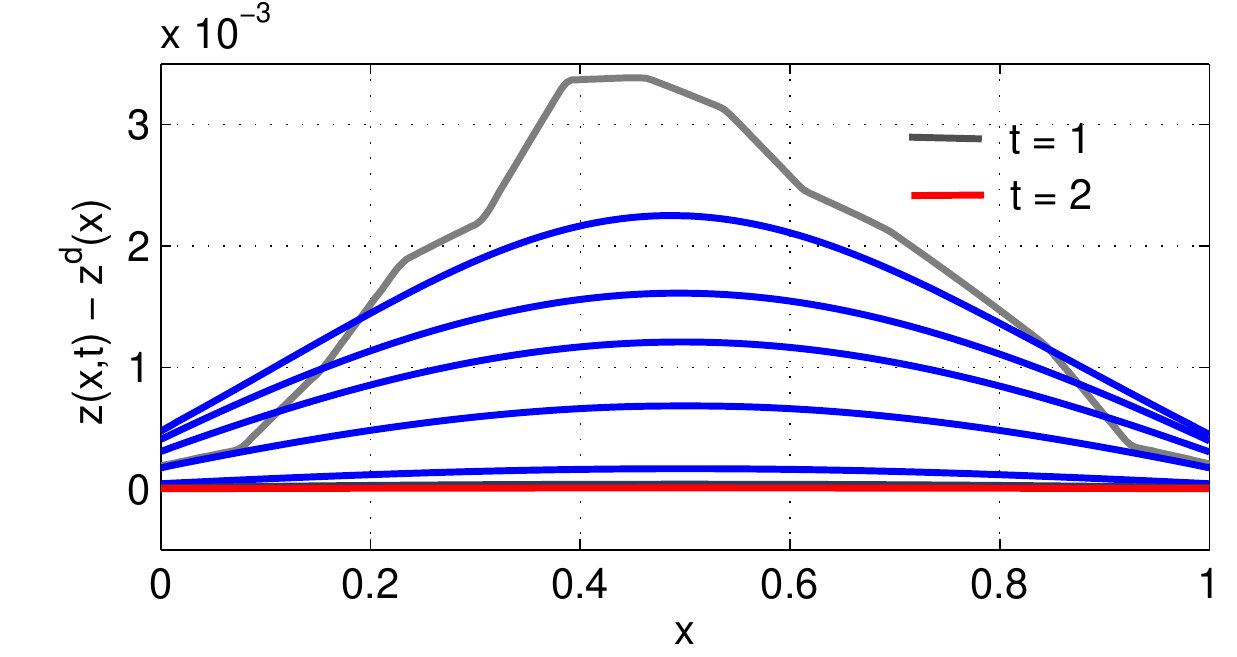}}
  \caption{Snapshot of regulation errors:
           (a) errors for $t\in [0,1]$;
           (b) errors for $t\in [1,2]$.}
  \label{fig: errors}
\end{figure}

%%%%%%%%%%%%%%%%%%%%%%%%%%%%%%%%%%%%%%%%%%%%%%%%%
\begin{figure}[thpb]\
  \centering
  \subfigure[]{\label{fig:Controls_1_6}  \includegraphics[scale=.45]{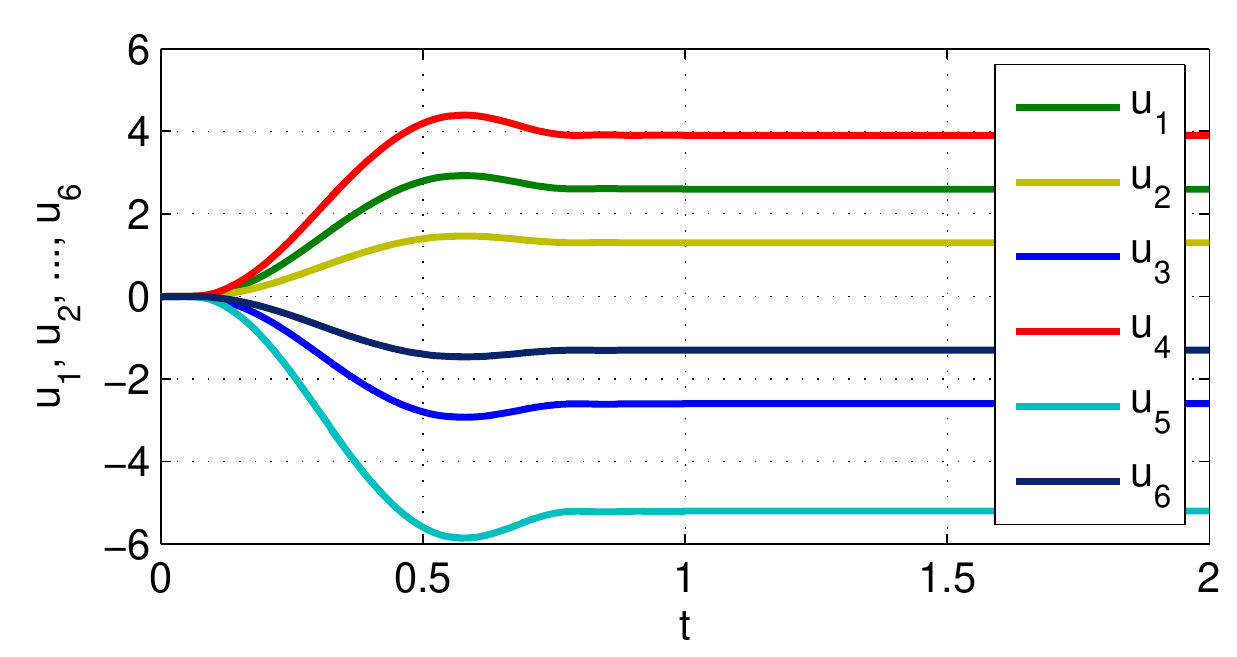}}
  \subfigure[]{\label{fig: Controls_7_12}\includegraphics[scale=.45]{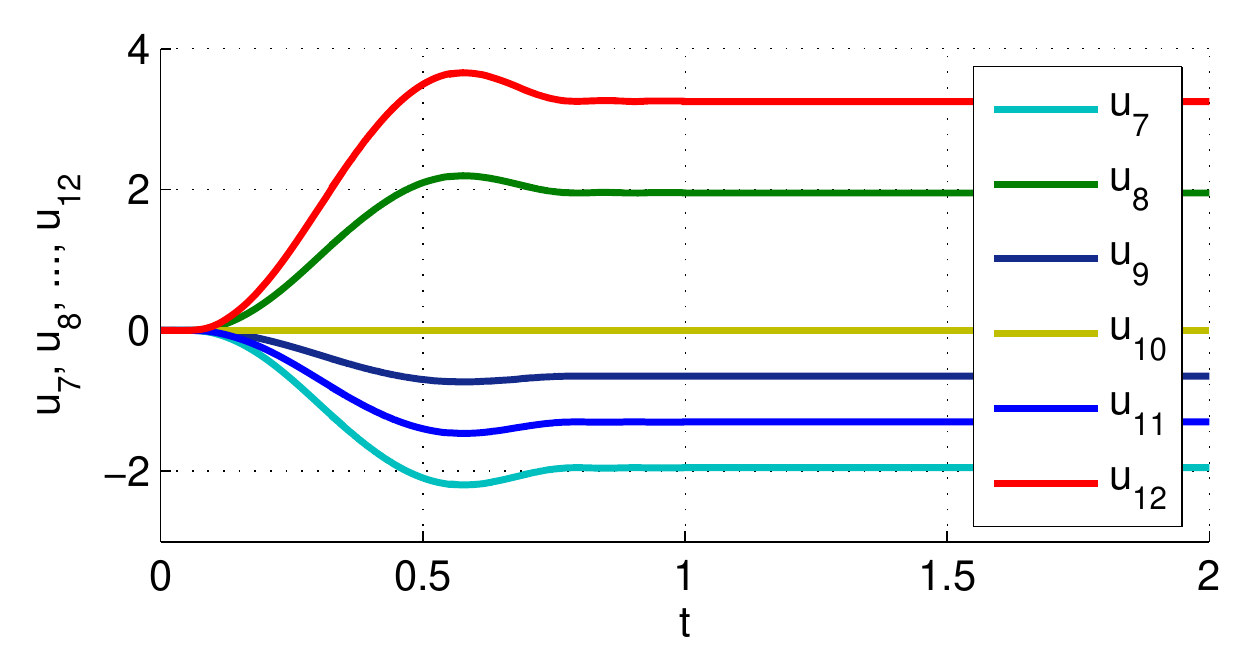}}
  \caption{Control signals:
           (a) $u_1$ to $u_6$;
           (b) $u_7$ to $u_{12}$.}
  \label{fig: Controls}
\end{figure}

%%%%%%%%%%%%%%%%%%%%%%%%%%%%%%%%%%%%%%%%%%%%%%%%%%%%%%%%%%%%%%%%%%%%%
\section{Conclusion}\label{Sec: Conclusion}
This paper presented a solution to the problem of set-point control of heat distribution with in-domain actuation described by an inhomogeneous parabolic PDE. To apply the paradigm of zero dynamic inverse, the system is presented in an equivalent \emph{parallel connection} form. The technique of flat systems is employed in the design of dynamic control and motion planning. As the control with multiple in-domain actuators results in a MIMO problem, a Green's function-based reference trajectory decomposition is introduced, which considerably simplifies the design and the implementation of the developed control scheme. The convergence and solvability analysis confirms the validity of the control algorithm and the simulation results demonstrate the viability of the proposed approach. Finally, as both ZDI design and flatness-based control can be carried out in a systematic manner, we can expect that the approach developed in this work may be applicable to a broad class of distributed parameter systems.

%%%%%%%%%%%%%%%%%%%%%%%%%%%%%%%%%%%%%%%%%%%%%%%%%%%%%%%%%%%%%%%%%%%%%%%%%%%%%%%%
\section*{Acknowledgements}
This work is supported in part by a NSERC (Natural Science and Engineering Research Council of
Canada) Discovery Grant. The first author is also supported in part by the Fundamental Research Funds for the Central Universities (\#682014CX002EM).

\section*{References}

\bibliography{mybibfile}

\end{document}